\documentclass [11pt]{article}
\usepackage{amsthm}
\usepackage{amsmath}
\usepackage{amssymb}
\usepackage{enumerate}
\usepackage{epsfig}

\newcounter{contador}
\newtheorem{propo}[contador]{Proposition}
\newtheorem{teo}[contador]{Theorem}
\newtheorem{lem}[contador]{Lemma}
\newtheorem{nota}[contador]{Remark}

\newtheorem{corol}[contador]{Corollary}

\newcommand{\rec}{\noindent} 
\renewcommand{\qed}{\, \hfill\rule[-1mm]{2mm}{3.2mm}} 

\newcommand{\Qm}{{\mathbb{Q}^+}}

\newcommand{\K}{{\mathbb K}}
\newcommand{\R}{{\mathbb R}}

\newcommand{\N}{{\mathbb N}}
\newcommand{\Q}{{\mathbb Q}}
\newcommand{\Z}{{\mathbb Z}}

\title{Rational Periodic Sequences for the Lyness
Recurrence}

\author{Armengol Gasull$^{(1)}$ V\'{\i}ctor Ma\~{n}osa $^{(2)}$ and
Xavier Xarles$^{(1)}$
\\*[.1truecm]
{\small \textsl{$^{(1)}$ Dept. de Matem\`{a}tiques, Facultat de
Ci\`{e}ncies,}}
\\*[-.25truecm] {\small \textsl{Universitat Aut\`{o}noma de Barcelona,}}
\\*[-.25truecm] {\small \textsl{08193 Bellaterra, Barcelona, Spain}}
\\*[-.25truecm] {\small \textsl{gasull@mat.uab.cat, xarles@mat.uab.cat}}
\\*[-.25truecm]
\\*[-.25truecm] {\small \textsl{$^{(2)}$ Dept. de Matem\`{a}tica Aplicada III (MA3),}}
\\*[-.25truecm] {\small \textsl{Control, Dynamics and Applications Group (CoDALab)}}
\\*[-.25truecm] {\small \textsl{Universitat Polit\`{e}cnica de Catalunya (UPC)}}
\\*[-.25truecm] {\small \textsl{Colom 1, 08222 Terrassa, Spain}}
\\*[-.25truecm] {\small \textsl{victor.manosa@upc.edu}}}
\date{}

\begin{document}
\maketitle


\begin{abstract} Consider the celebrated Lyness recurrence $
x_{n+2}=(a+x_{n+1})/x_{n}$ with $a\in\Q$. First we prove that there
exist initial conditions and values of $a$ for which it generates
periodic sequences of rational numbers with prime periods
$1,2,3,5,6,7,8,9,10$ or $12$ and that these are the only periods
that rational sequences $\{x_n\}_n$ can have. It is known that if we
restrict our attention to positive rational values of $a$ and
positive rational initial conditions the only possible periods are
$1,5$ and $9$. Moreover  1-periodic and 5-periodic sequences are
easily obtained. We prove that  for infinitely many positive values
of $a,$ positive 9-period rational sequences  occur. This last
result is our main contribution and answers an open question left in
previous works of Bastien \& Rogalski and Zeeman. We also prove that
the level sets of the invariant associated to the Lyness map is a
two-parameter family of elliptic curves that is a universal family
of the elliptic curves with a point of order $n, n\ge5,$ including
$n$ infinity. This fact implies that the Lyness map is a universal
normal form for most birrational maps on elliptic curves.
\end{abstract}

\rec {\sl 2000 Mathematics Subject Classification:} \texttt{39A20,
39A11,14H52.}

\rec {\sl Keywords:} Lyness difference equations, rational points
over elliptic curves, periodic points, universal family of elliptic
curves.
\newline

\section{Introduction and main results}

The dynamics of the Lyness recurrence
\begin{equation}\label{Lynesseq}
x_{n+2}=\displaystyle{\frac{a+x_{n+1}}{x_{n}}},
\end{equation}
specially when  $a>0,$ has focused the attention of many researchers
in the last years and it is now completely understood in its main
features after the independent research of Bastien \& Rogalski
\cite{BR} and Zeeman \cite{Z}, and the later work of Beukers \&
Cushman \cite{BC}. See also \cite{BGS,ER}. In particular all
possible periods of the recurrences generated by \eqref{Lynesseq}
are known and for any $a\notin\{0,1\}$ infinitely many different
prime periods appear.

However there are still some open problems concerning the dynamics
of rational points. With the computer experiments in mind, and
following \cite{BR,Z}, it is interesting to know the existence of
rational periodic sequences. The Lyness map
\begin{equation}\label{map} F_a(x,y)=(y,(a+y)/x),
\end{equation}
associated to \eqref{Lynesseq}, leaves invariant the  elliptic
curves\footnote{For some concrete  values of $h$ and $a$ the curve
is not elliptic. We study these values separately.}
\[
C_{a,h}:=\{(x+1)(y+1)(x+y+a)-hxy=0\}
\]
and the map action can be described in terms of the group law action
of them. In consequence several tools for studying the rational
periodic orbits on them are available. In particular, from Mazur's
Torsion Theorem (see for example \cite{ST}), we know that, under the
above hypotheses, the rational periodic points can only have (prime)
periods $1,2,\ldots,9,10$ and $12.$ Our first result proves that
almost all these periods appear for the Lyness recurrence for
suitable $a\in\Q^+$ and rational initial conditions.

\begin{teo}\label{teo}
For any $n\in\{1,2,3,5,6,7,8,9,10,12\}$ there are
$a\in\Q^+\cup\{0\}$ and rational initial conditions $x_0,x_1$ such
that the sequence generated by \eqref{Lynesseq} is $n$-periodic.
Moreover these values of $n$ are the only possible prime periods for
rational initial conditions and $a\in\Q$.
\end{teo}

Notice that the value $n=4$ is the only one in Mazur's list that is
not in the list given in the theorem. Following \cite{Z} it is
possible to interpret  that this period corresponds to the case
$a=\infty,$ see Remark~\ref{infinity} in Section~\ref{totq}.

Concerning the rational periodic points in $\Qm\times\Qm$ for the
Lyness map with $a>0,$ it is proved in \cite{BR} that it only can
have periods $1, 5$ and $9.$

Taking $a=n^2-n$ and $x_0=x_1=n\in\N$ we obtain trivially 1-periodic
integer sequences. The existence of positive rational periodic
points of period 5 is well known and simple: they only exist when
$a=1$ and in this case all rational initial conditions give rise to
them because the recurrence \eqref{Lynesseq} is globally 5-periodic.
On the other hand, as far as we know, the case of period 9 has
resisted all previous analysis. In particular, the
Conjecture\footnote{Conjecture 1 of Zeeman was about the
monotonicity of certain rotation number function associated to the
invariant ovals of the Lyness map and was proved in \cite{BC}.}~2 of
Zeeman, \cite{Z} says that there are no such points and their
existence is left in Problem\footnote{Problem 1 of \cite{BR} is
about rational 5-periodic points and it is recalled and solved in
Section~\ref{periode5}.} 1 bis of \cite{BR} as an open question.

We prove that there are some values of $a$ for which the Lyness
recurrences \eqref{Lynesseq} have positive rational periodic
sequences of period 9 and, even more, that this happens for
infinitely many values of $a,$ see more details in Theorem
\ref{teoA}.

It is known, see again \cite{BR,Z}, that the periodic points of
period 9 of $F_a$ are on the elliptic curve
\begin{equation}\label{9p}
a(x+1)(y+1)(x+y+a)-(a-1)(a^2-a+1)xy=0,\end{equation} and they have
positive coordinates only when $a>a_*\simeq 5.41147413$ where
$a_{*}$ is the biggest root of ${a}^{3}-6\,{a}^{2}+3\,a+1=0,$ see
also Subsection~\ref{grouplaw}.

By using MAGMA (\cite{BCP})\footnote{It is also possible to use SAGE
(\cite{Sage}).}, and after several trials, we have found some
positive rational points on the above curve proving that the
Zeeman's Conjecture~2 has a negative answer. The simplest one that
we have obtained is $(a;x,y)=(7;3/2,5/7)$. Notice that the
sequence~\eqref{Lynesseq}, taking $a=7$ and the initial condition
$x_0=3/2$, $x_1=5/7,$ gives
\[
\frac 32,\, \frac 57,\, \frac{36}7,\,
17,\,\frac{14}{3},\,\frac{35}{51},\,\frac{28}{17},\,\frac{63}5,\,\frac{119}{10},\,
\frac32,\,\frac57,\,\ldots
\]
Other positive rational points that we have found are
\[
\left(11;\frac{29}{82},\frac{19}{22}\right),\quad\left(13;\frac
{1584676}{61133}, \frac{335937}{856427}\right),\quad \left(19; \frac
{4259697} {16150}, \frac {5178617} {168283}  \right)
\]
%
and many others with much bigger entries. Our main result proves
that there are infinitely many positive rational values of the
parameter $a$ giving rise to $9$--periodic positive rational orbits.

\begin{teo}\label{teoA}
There are infinitely many values $a\in\Qm$ for which there exist
initial conditions $x_0(a),x_1(a)\in \Qm$ such that the sequence
given by the Lyness recurrence \eqref{Lynesseq} is $9$--periodic.
Furthermore, the closure of these values of $a$ contains the real
interval $[a_1,+\infty)$, where $a_1\simeq 5.41147624$ is the
biggest root of $a^3-\frac{2019}{529}a^2-\frac{777}{92}a-1.$
\end{teo}

Notice that there is a small gap between the values $a_*$ and $a_1$
where we do not know if there are or not rational values of $a$ for
which the Lyness recurrence has positive rational periodic orbits of
period 9. As we will see in the proof of the theorem the gap is
provoked by our approach and it seems to us that it is not intrinsic
to the problem, see the comments in Section~\ref{periode9}, after
the proof of the theorem.

From the classical results of Mordell, see for example
\cite[Ch.VIII]{S}, it is well known that the set of rational points
on an elliptic curve $E$ over $\Q,$ together with the point at
infinity, form an additive group $E(\Q)$ and
\[
E(\Q)\cong \Z^r\times \Phi,
\]
where $r\in\N$ is called the rank of $E(\Q)$ and $\Phi$ is the
torsion of the group. Notice that $r$ is a measure of the amount of
rational points that the curve contains. The torsion part $\Phi$ is
well understood from the results of Mazur already quoted, and it
contains at most 16 points. For the elliptic curve \eqref{9p} it is
easy to check that $\Phi=\Z/9\Z.$ Fixed any $\Phi_0,$ among all the
allowed possibilities, it is not known if the rank of the elliptic
curves having $E(\Q)\cong \Z^r\times \Phi_0,$ for some $r\in\N$, has
an upper bound, but it is believed that it has not (see the related
Conjecture VIII.10.1 in \cite{S}). As far as we know, nowadays when
$\Phi_0= \{0\}$ the highest known rank is greater or equal than~28
while when $\Phi_0=\Z/9\Z $ the highest known rank is~4, see
\cite{W}. We have found values of $a$ for which the algebraic
curve~\eqref{9p} has ranks $0,1,2,3$ or $4$. For instance $0$
appears for $a=6$ and the value 4 happens for $a=408/23.$ A key
point to obtain these results  is the following theorem,
 which extends the results of \cite{Beau}, given for the cases
 of torsion points with order 5 or 6, and proves the universality of
 the level curves invariant for the Lyness map.

\begin{teo}\label{nft}(Lyness normal form)
The family of elliptic curves $C_{a,h}$,
\[
(x+1)(y+1)(x+y+a)-hxy=0,
\]
 together with the points
${\cal O}=[1:-1:0]$ and $Q=[1:0:0]$, is the universal family of
elliptic curves with a point of order $n,$  $n\ge5$ (including
$n=\infty$). This means that for any elliptic curve $E$ over any
field $\K$ (not of characteristic 2 or 3) with a point $R\in E(\K)$
of order $n$, there exist unique values $a_{(E,R)},h_{(E,R)}\in \K$
and a unique isomorphism between $E$ and $C_{a_{(E,R)},h_{(E,R)}}$
sending the zero point of $E$ to ${\cal O}$ and the point $R$ to
$Q$. Moreover, given a finite $n$, the relation between $a$ and $h$
can be easily obtained from the computations made in
Subsection~\ref{grouplaw}.
\end{teo}

From a dynamical viewpoint the above result implies that the Lyness
Map is an affine model for most birrational maps on elliptic curves.
This result is similar to the one described by   Jogia, Roberts and
Vivaldy in \cite[Theorem 3]{JRV}, where they use the Weierstrass
normal form. Moreover, notice that  in our result the sum operation
takes the extremely easy form $\tilde{F}([x:y:z])=[xy:az^2+yz:xz]$
(see Section \ref{grouplaw}).

Once for some value $a\in\Q^+$ a periodic orbit in $\Q^+\times\Q^+$
of period $9$ for $F_a$ is obtained, it is not difficult to obtain
infinitely many different 9-periodic orbits by using the group law
on the curve. For instance, for a given $a,$ we know that if a point
$P=(x,y)$ is on \eqref{9p} then $(2k+1)P, k\in\Z$ is also on it. In
particular for $a=7$ if we write $3P=(z,w)$ we get
\[
z=\frac{2(6727x+913y+913)(90272x-415y-2905)}
{(5583410x^2+858451819xy-28403347y^2-187465799x-227226776y-198823429)}
\]
and $w$ can be obtained for instance by plugging $x=z$ in \eqref{9p}
and choosing a suitable solution~$y.$ Taking $P=(3/2,5/7)$ we find
\[
3P=\left(\frac{260143588}{23256135},\frac{337001111}{246029869}\right),
\]
which also gives rise to a 9-periodic orbit in $\Q^+\times\Q^+$.
Taking now $3P$ as starting point for the procedure we obtain a
new initial condition for a 9-periodic orbit:
\[
9P=\left(
\frac{3147471926986755321149021}{226091071032606625830925},\frac
{891522142852888213265718}{85174628288506877231975}\right),
\]
and so on. In general we have the following result, see
Subsection~\ref{grouplaw}.

\begin{propo}\label{propob}
If $a\in\Qm$ is a value for which there exist an initial condition
$(x_0,x_1)\in \Qm\times\Qm$ such that the sequence~\eqref{Lynesseq}
is $9$--periodic, then there exist infinite many different rational
initial conditions giving rise to $9$--periodic sequences. Moreover
the points corresponding to these initial conditions fill densely
the elliptic curve (\ref{9p}).
\end{propo}

Observe that such a point $(x_0,x_1)\in \Qm\times\Qm$ always gives
rise to a point on the elliptic curve given by \eqref{9p} which is
not a torsion point, because all the $9$-torsion points are not in
$\Q^+\times\Q^+,$ see   again Subsection~\ref{grouplaw}.

The case of 5-periodic points is studied in Section~\ref{periode5},
proving that for $a=1$ there are some elliptic curves $C_{1,h},
h\in\Q,$  with rank $0$ and, as a consequence, ovals without points
with  rational coordinates.

Finally, last section is devoted to give rational values of $a$ for
which the Lyness map $F_a$ has as many periods as possible, of
course with rational initial conditions.

The structure of this paper is the following. In Subsection
\ref{grouplaw} we recall some known results which describe the
action of the Lyness map in terms of a linear translation over
elliptic curves. Subsection~\ref{nf} is devoted to prove
Theorem~\ref{nft} about the Lyness normal form. In
Section~\ref{totq} we prove Theorem~\ref{teo}, while the proof our
main result, Theorem \ref{teoA}, and all our outcomes on rational
9-periodic points of \eqref{map} are presented in Section
\ref{periode9}. Section~\ref{periode5} studies the number of
rational points on the elliptic curves invariant for $F_1$ and the
last section deals with the problem of finding a concrete $F_a$ with
as many periods as possible.

\section{Preliminary results}

\subsection{Lyness recurrence from group law's action
viewpoint}\label{grouplaw} The results of this subsection are well
known, we refer the reader to references \cite{BR,JRV,Z} for
instance to get more details.

As mentioned above, the phase space of the discrete dynamical
system defined by the map \eqref{map} is foliated by the family
curves
$$
C_{a,h}:=\{(x+1)(y+1)(x+y+a)-hxy=0\}.
$$

This family is formed by elliptic curves except for a few values of
$h$. When $h=0$ it is  a product of straight lines; when $h=a-1$,
$a\ne1$, it is the formed  by a straight line and an hyperbola; and
when
\begin{equation}\label{hchc}
h=h^{\pm}_c:=\frac {2\,a^2+10\,a-1 \pm (4\,a+1)\sqrt {4\,a+1}}{2\,a},
\end{equation}
with $h_c^{\pm}\ne a-1,0$, it is a rational cubic, having   an
isolated real singularity. In fact the values $h^\pm_c$ correspond
to the level sets containing the fixed points of $F_a$,
$((1\pm\sqrt{4a+1})/2, (1\pm\sqrt{4a+1})/2)$.

The first quadrant $Q_1=\{(x,y), x>0, y>0\}$ is invariant under the
action of $F_a$ and it is foliated by ovals with energy $h> h^+_c.$

In summary, on most energy levels, $F_a$ is a birational map on an
elliptic curve and therefore it can be expressed as a linear action
in terms of the group law of the curve \cite[Theorem 3]{JRV}.
Indeed, taking homogeneous coordinates on the projective plane
$P\R^2$ the curves $C_{a,h}$ have the form
$$
\tilde{C}_{a,h}:=\{(x+z)(y+z)(x+y+az)-hxyz=0\},
$$
and $F$ can be seen as the map $\tilde{F}([x:y:z])=[xy:az^2+yz:xz]$.
Taking the point ${\cal O}=[1:-1:0]$ as the neutral element, each
elliptic curve $\tilde{C}_{a,h}$ is an abelian group with respect
the sum defined by the usual \textsl{secant--tangent chord process}
(i.e. if a line intersects the curve in three points $P$,$Q$, and
$R$ then $P+Q+R=\cal{O}$). According to these operation the Lyness
map can be seen as the linear action
\begin{equation}\label{accio}
\tilde{F}:P\longrightarrow P+Q,
\end{equation}
where $Q=[1:0:0]$. When $a(a-1)\ne0$ it is not difficult to get that
$2Q=[-1:0:1],$ $3Q=[0:-a:1],$
\begin{align*}
4Q&=\left[-a:\frac{ah-a+1}{a-1}:1\right],\\
5Q&=\left[\frac{ah-a+1}{a-1}:\frac{-a^2-ah+2a-1}{a(a-1)}:1\right],\\
6Q&=\left[\frac{-a^2-ah+2a-1}{a(a-1)}:\frac{a^3-2a^2-ah+2a-1}{a(ah-a+1)}:1\right],\\
7Q&=\left[\frac{-a^2-ah+2a-1}{a(a-1)}:\frac{-a^4h+a^3h+a^3+a^2h-3a^2-ah+3a-1}{a^3h-a^3+a^2h^2-3a^2h+3a^2+2ah-3a+1}
:1\right],
\end{align*}
$-Q=[0:1:0],$ $-2Q=[0:-1:1],$ $-3Q=[-a:0:1],$
\begin{align*}
-4Q&=\left[\frac{ah-a+1}{a-1}:-a:1\right],\\
-5Q&=\left[\frac{-a^2-ah+2a-1}{a(a-1)}:\frac{ah-a+1}{a-1}:1\right],
\end{align*}
see also \cite{BR}.

In terms of the recurrence, the linear action (\ref{accio}) can be
seen as follows: taking the initial conditions $P:=[x_0:x_1:1]$ then
$[x_{n+1}:x_{n+2}:1]=P+(n+1)Q$, where $+$ is the group operation.
Notice also that, from this point of view, the condition of
existence of rational periodic orbits is equivalent to the condition
that $Q$ is in the torsion of the group given by the rational points
of $\tilde{C}_{a,h},$ which, as we have already commented, is
described by Mazur's Theorem.

Hence, from the above expressions of $k Q$  we can obtain the values
of $h$ corresponding to a given period. For instance, for period 9
we impose that $4Q=-5Q$, or equivalently $9Q={\cal O},$ which gives
$-a=(-a^2-ah+2a-1)/(a(a-1))$. From this equality we have that
\[
h=\frac{(a-1)(a^2-a+1)}{a},
\]
which  corresponds to the elliptic curve~\eqref{9p}. For these
values of $h$ the points corresponding to the torsion subgroup are:
\begin{align*} Q=&[1:0:0],\,[-1:0:1],\,[0:-a:1],\, [-a:a(a-1):1],\\
&[a(a-1):-a:1],\, [-a:0:1],\,[0:-1:1],\, [0:1:0],\, {\cal
O}=[1:-1:0].\end{align*}


It is also important from a dynamical point of view the following
well know property of the secant--tangent chord process defined on
any real non-singular elliptic curve $E$. Let $P$ be a point of $E$
such that $kP,$ for $k\in\Z,$ is never the neutral element ${\cal
O}$. Recall that when the elliptic curve is defined over $\Q$ there
are at most sixteen points $P$ with rational entries not satisfying
this property due the Mazur classification of the torsion subgroup
of $E(\Q)$. Then the adherence of the set $\{kP\}_{k\in\Z}$ is:
\begin{itemize}
\item either all the curve $E$, when $P$ belongs to the
connected component  of $E$ which does not contains the neutral
element $\cal O$ ; or
\item the  connected component containing $\cal O$
when $P$ belongs to it.
\end{itemize}
This result is due to the fact that there is a continuous
isomorphism between $E\cup{\cal O}$ with this operation and the
group $\{e^{it}\,:\,t\in[0,2\pi)\}\times \{1,-1\}$, with the
operation $(u,v)\cdot(z,w)=(uz,vw),$ see for instance Corollary
2.3.1 of \cite[Ch. V.2]{S0}. Clearly, by using this construction,
Proposition~\ref{propob} follows.

Notice that when an elliptic curve is given by $C_{a,h}$ the bounded
component never contains~$\cal O$.

\subsection{A new normal form for elliptic curves}\label{nf}

\begin{proof}[Proof of Theorem~\ref{nft}]
It is known that any elliptic curve having a point $R$ that is not a
2 or a 3 torsion point can be written in the so called Tate normal
form
\[
Y^2+(1-c)XY-bY=X^3-bX^2,
\]
where $R$ is sent to $(0,0)$, see \cite[Ch. 4, Sec. 4]{Hu}. Our
proof will follow by showing that the curves $C_{a,h}$ can be
transformed into the ones of the  Tate normal form. In projective
coordinates these curves write as
\[
Y^2Z+(1-c)XYZ-bYZ^2=X^3-bX^2Z
\]
and the curves $C_{a,h}$ as
\[
(x+z)(y+z)(x+y+az)-hxyz=0.
\]
With the change of variables
\[
X=-\frac{b}{c+1}z,\quad Y=-\frac{bc}{c+1}(y+z),\quad
Z=-\frac{c}{c+1}(x+y)-z
\]
and the relations
\[
h=-\frac{b}{c^2},\quad a=\frac{c^2+c-b}{c^2},
\]
both families of curves are equivalent and the theorem follows.
Observe that the case $c=0$ corresponds to a curve with a
$4$--torsion point. As we will see in the proof of
Theorem~\ref{teoA} one advantage of this normal form is that it is
symmetric with respect to $x$ and $y.$
\end{proof}

From Theorem~\ref{nft} we have that all the known results on
elliptic curves with a point of order greater than 4 can be applied
to the corresponding Lyness curves. In particular we find inside
$C_{a,h},$ the curves with high rank and prescribed torsion given in
\cite{W} or we can use the list of ``Elliptic Curve Data" for curves
in Cremona form (\cite{Cr}), taking advantage of the MAGMA or SAGE
softwares that allow to identify a given elliptic curve in it.

\section{Possible periods for rational points}\label{totq}

\subsection{The non-elliptic curves case}\label{nonelli}

As was explained in Subsection~\ref{grouplaw} the curves $C_{a,h}$
are elliptic for all values of $a$ and $h$ except for $h\in
H:=\{0,a-1,h^{\pm}_c\}$, with $h^{\pm}_c$  given in (\ref{hchc}). On the
curves corresponding to  these values  there could be, for the
rational periodic orbits,  some periods that are not in the list
given by Mazur's theorem. In this section we prove that  no new
period appears.

\begin{lem}\label{genere0} The  periods of the rational periodic orbits of $F_a$
lying on the curves $C_{a,h}$ for $h\in H$ are $1,2,3,6,8$ and 12.
\end{lem}

\begin{proof} It is well known that the case $a=0$ is globally 6-periodic.
So from now on we consider that $a\ne0.$ We start the study of the
rational cubic curves $C_{a,h^{\pm}_c}$, where $h^{\pm}_c$ are given in
(\ref{hchc}) and moreover $h^\pm_c\not\in\{0,a-1\}$. By setting
$b=\pm\sqrt{4a+1}$, we have
$$\frac{(b+ 3)^3}{4(b+1)}=
\begin{cases}h^+_c \quad\mbox{ when }
\quad b\ge0,\quad b\not=1,\\h^{-} _c \quad\mbox{ when }\quad
b\le0,\quad b\notin\{-1,-2,-3\}.
\end{cases}
$$
Note that $a=0$ implies that $b=\pm1$; $h_c^\pm=a-1\ne0$ implies
that $b=-2$; and $h_c^\pm=0$ implies that $b=-3$. Observe also that
since $a$ and $h_c^\pm$ are in $\Q$ then $b\in\Q.$

By using again a computer algebra software we obtain the following
joint parametrization of both curves $C_{a,h^{\pm}_c}$:

$$t\to\left(x(t),y(t)\right)=\left(
{\frac { \left( 3t+tb-2 \right)  \left( 2tb+4t-b-1 \right) }{
2\left( b+1 \right)  \left( t-1 \right) }},-{\frac { \left(
3t+tb-b-1 \right)  \left( 2tb+4t-3-b
 \right) }{2t \left( b+1 \right) }}\right).$$
Moreover
\begin{equation}\label{m}
t=\frac{x(t)-(b+1)/2}{x(t)+y(t)-(b+1)}.
\end{equation}

On each curve $C_{a,h^{\pm}_c}$ the Lyness map $F_a$ can be seen as
$$
\left.F_a\right|_{C_{a,h^{\pm}_c}}:\big(x(t),y(t)\big)\longrightarrow
\left(y(t),\frac{(b^2-1)/4+y(t)}{x(t)}\right)=
\big(x(f(t)),y(f(t))\big),
$$
where $f(t)$ has to be determined. By using \eqref{m} we obtain that
$f$ is the linear fractional transformation
$$
f(t)=\frac{t- (b+ 1)/(2b+ 4)}{t},
$$
where recall that $b\in\Q.$ Therefore  the dynamics of $F_a$ on each
of the curves $C_{a,h^{\pm}_c}$ is completely determined by the dynamics
of the maps $f(t)$.

It is a well-known fact that if a linear fractional map,
$g(t)=(At+B)/(t+D)$, has a periodic orbit of prime period $p$,
$p>1,$ then it is is globally $p$-periodic. Moreover this happens if
and only if either:
\begin{itemize}
\item  $\Delta:=(D-A)^2+4B> 0$ and $A=-D$ and  in this case $g$ is
$2$-periodic; or

\item $\Delta<0$ and $\xi:=(A+D-\sqrt{|\Delta|}\,i)/(A+D+\sqrt{|\Delta|}\,i)$ is a
  primitive $p$-root of  the unity and in this case $g$ is $p$-periodic.
\end{itemize}

   Hence, apart of the fixed points,  the maps $f(t)$ can  have
  $p$-periodic solutions, $p>1$, if and only if
  $$
\Delta=-\frac{b}{b+2}<0 \mbox{ and }
\xi=\frac1{b+1}-\frac{b+2}{b+1}\sqrt{\frac{b}{b+2}}i \mbox{ is a
primitive }p\mbox{-root of the unity.}
  $$

The condition that $\xi$ is a primitive $p$-root of the unity,
implies that
$$
\cos\left(2\pi\frac{q}{p}\right)=\frac{1}{b+ 1}\in\Q, \mbox{ for
some } q\in\Z.
$$

It is also a well-known fact that the only rational values of
$\cos(x)$, where $x$ is a rational multiple of $\pi,$ are $0$,$\pm
1/2$,$\pm 1$ (see \cite[Theorem 6.16]{NZM} or \cite{O}). This fact
implies that $b\in\{-3,-2,0,1\}.$ The only allowed valued is $b=0$,
which implies that $\Delta=0,$ and so the corresponding map $f$ is
not periodic. Hence, only the fixed points of $F_a$ appear on this
family of curves.

Concerning the case $h=0$, observe that
$C_{a,0}=\{(x+1)(y+1)(a+x+y)=0\}$. When $a=1$ there are no periodic
orbits on this level set. When $a\ne1,$  the three straight lines
forming this set are mapped one into the other in cyclical order by
$F_a,$ and so they are invariant under $ F_a^3.$ The cyclical order
determined by $F_a$ is:
$$\{x+1=0\}\to\{a+x+y=0\}\to\{y+1=0\}\to\{x+1=0\}.$$ The
restriction of $F_a^3$ on $\{x+1=0\}$ is given by $$
  F_a^3(-1,y)=\left(-1,\frac{1-a}{y+a}\right).
$$
Hence the dynamics of $F_a^3$ is  determined by the dynamics of the
linear fractional map
$$
f(y)= \frac{1-a}{y+a}.
$$
One of its fixed points corresponds to the continua of three
periodic points of $F_a$ given in Table~1. By using again the
characterization of the periodicity of the linear fractional maps we
obtain that $f$ is periodic only when $a=0.$

It remains  to study the case $h=a-1, a\ne1$. In this situation
$C_{a,a-1}=\{(x+y+1)(a+x+y+xy)=0\}$ and $F_a$ sends the straight
line to the hyperbola and vice versa. So both level sets are
invariant under $ F_a^2.$ The restriction of $F_a^2$ on
$\{x+y+1=0\}$ is given by $$
  F_a^2(x,-1-x)=\left(\frac{-x+a-1}{x},\frac{1-a}{x}\right).
$$
Hence the dynamics of $F_a^2$ is determined by the linear fractional
map
$$
f(x)=\frac{-x+a-1}{x},
$$
which has  fixed points only when $a\ge3/4.$ They give rise to
2-periodic points of $F_a$ when $a>3/4$ and to a fixed point when
$a=3/4.$ Arguing as in the case $h=h_c^\pm$, the map $f$ is
$p$-periodic, $p\ge2,$ only when
  $$
\Delta=4a-3<0 \mbox{ and } \xi=\frac{1-2a-i\sqrt{3-4a}}{2a-2}\mbox{
is a primitive }p-\mbox{root of the unity.}
  $$
This happens if and only if $a<3/4$ and
$(1-2a)/(2a-2)\in\{0,\pm1/2,\pm1\},$ or equivalently when
$a\in\{1/2,2/3\}$ (recall that the case $a=0$ is already
considered). The case $a=1/2$ gives a $4$-periodic map $f$ and
$a=2/3$ a $6$-periodic map. These cases correspond, respectively, to
the existence of continua of $8$ and $12$ periodic points for $F_a$
on $C_{a,a-1}$.
\end{proof}

\subsection{Proof of Theorem~\ref{teo}}

\begin{proof}[Proof of Theorem~\ref{teo}]
From Lemma~\ref{genere0} we know that when $h\in H$ the possible
periods on $C_{a,h}$ are in the list given in the statement. For
those points on the elliptic  curves $C_{a,h}$ for all values of $a$
and $h\notin H$ we can apply Mazur's theorem and we obtain that the
only possible periods are the ones of the statement together with
the period 4. The points of (prime) period
 4 can be discarded  by observing that in Subsection~\ref{grouplaw}
 we prove that $4Q\ne\mathcal{O}.$
It is also possible to perform a direct study with resultants of the
system $F^4_a(x,y)=(x,y)$. From this study we get that that its only
solutions are  the ones corresponding to fix or
 2-periodic  points.

These results together with the ones presented on  Table 1 prove the
theorem.
\end{proof}

\vspace{0.2cm}

\begin{center}
\begin{tabular}{|c||c|c|c|c|}
  \hline
 Period & $a$ & $x_0$ & $x_1$& Comments\\
 \hline
  1 & $u^2-u$ & $u$ & $u$ & $u\in\Q\setminus\{0\}$ \\
 2 & $u^2+u+1$ & $u$ & $-u-1$ & $u\in\Q\setminus\{-1,0\}$\\
 3 & $a$ & $-1$ & $-1$ & $a\in\Q\setminus\{1\}$\\
 5 & 1 & $x_0$ & $y_0$ & Almost for all $x_0$, $y_0$ in $\Q$\\
 6 & 0 & $x_0$ & $y_0$ & Almost for all $x_0$, $y_0$ in $\Q$\\
 7 & $\frac{u^2-1}{2u-1}$& $\frac{u^2-1}{u^2-u+1}$ & $-x_0$ & Almost for all $u$ in $\Q$\\
 8 & $\frac{u^2-1}{u^2+2u-1}$ & $\frac{u^2-1}{u^2+1}$ & $-x_0$ & Almost for all $u$ in $\Q$ \\
 9 & 7 & 3/2 & 5/7 & See also Theorem~\ref{teoA}\\
 10 & 3/2 & -2 & 3/5 & There are infinitely many\\
 12 & 12/13 & -4/9 & -10/13 & There are infinitely many\\
  \hline
\end{tabular}
\vspace{0.5cm} \nobreak\\ Table 1. Examples of rational periodic
sequences for recurrence~\eqref{Lynesseq}. See also
Remark~\ref{taula}.
\end{center}

\vspace{0.2cm}

Next remark collects some comments on the results presented in Table
1.

\begin{nota}\label{taula} (i) The continua of sequences of periods 1 and 2 are not
the most general ones, we have chosen simple one-parameter families.

(ii) It is already known that there exist one-parameter families of
elliptic curves having points of 7 (or 8) torsion and rank 1.  By
using Theorem~\ref{nft}  we know that they should appear also in the
Lyness normal form. Nevertheless, we have got them by a direct
study, only searching points satisfying $x+y=0.$

(iii) The values corresponding to periods 10 and 12 are obtained
with MAGMA. By using similar arguments to the ones used to study the
case of period 9 we can prove the existence of rational periodic
points, with these periods,  for infinitely many rational values of
$a.$ It is not known the existence of a continua of them.
\end{nota}

\begin{nota}\label{infinity}
By introducing  $y_n:=x_n/\sqrt{a}, n\in\N,$ the Lyness recurrence
\eqref{Lynesseq} writes as $
y_{n+2}=\left({1+{y_{n+1}}/{\sqrt{a}}}\right)/{y_n}. $ When $a$
tends to infinity we obtain the recurrence $y_{n+2}=1/y_n$, which is
globally 4-periodic.
\end{nota}

\section{Proof of Theorem~\ref{teoA}}\label{periode9}

In the proof of Theorem \ref{teoA} we will use the following
Lemma, see for example \cite{AM}.

\begin{lem}\label{lemaxarles}
The curve $K^2=A^4+w_2A^2+w_1A+w_0$ is isomorphic to the elliptic
curve
$$
Y^2=X^3-\left(\frac{w_2^2}{48}+\frac{w_0}{4}\right)X+\frac{w_1^2}{64}+\frac{w_2^3}{864}-\frac{w_0w_2}{24},
$$
where the change of variables is given by
$$
X=\frac{1}{2}\left(A^2+K+\frac{w_2}{6}\right),\quad
Y=\frac{A}{2}\left(A^2+K+\frac{w_2}{2}\right)+\frac{w_1}{8}.
$$
\end{lem}

\noindent {\it Proof of Theorem \ref{teoA}.} As we have already
explained in the introduction, following \cite{BR,Z} we already know
that all the real \textsl{positive} initial conditions that give
rise to 9-periodic recurrences~\eqref{Lynesseq} correspond to the
points of the  positive oval of the elliptic curves~\eqref{9p},
\[
S_a:=\{(a;x,y)\,:\,a(x+1)(y+1)(x+y+a)-(a-1)(a^2-a+1)xy=0,\\ x>0,
y>0, a>a_*\}.
\]

We want to find infinitely many points $(x(a),y(a),a)\in
(\Qm\times\Qm\times\Qm)\cap S_a$. As we will see bellow, we will
find first one point $(a;x,y)$, satisfying $x+y=23/4,$ and from it
we will construct infinitely many via a multiplication process.

We start by applying the transformation $S=x+y$, $P=xy$ to each
surface $S_a.$ We obtain,
\begin{equation*}
a \left( 1+S+P \right) \left( a+S \right) -\left( a-1 \right)
\left( {a}^{2}-a+1 \right) P=0,
\end{equation*}
or equivalently,
\begin{equation}\label{eqP}
P=\displaystyle{\frac {a \left( 1+S \right) \left( a+S \right)
}{a^3-3a^2+(2-S)a-1}}.
\end{equation}

It is easy to see that if $(S,P)\in \Qm\times\Qm$ and
$\Delta:=S^2-4P$ is a perfect square then the corresponding
$(x,y)\in \Qm\times\Qm.$ So, we want to find a value of $S$ such
that $\Delta$ has a suitable expression that facilitates to find
values of $a$ for which $\Delta$ is a perfect square.

By using \eqref{eqP}, $\Delta=S^2-4P$ writes as
$$
\Delta=\displaystyle{\frac{{S}^{3}a- \left( a^3-3\,a^2-2\,a-1
\right) {S}^{2} + \left( 4\,{a}^{2}+4\,a \right) S
+4\,{a}^{2}}{a^3-3a^2+(2-S)a-1}}.
$$

We will fix a value of $S$ such that the value of the discriminant
with respect $a$ of the denominator of the last expression
vanishes. The discriminant is $(4S-23)(S+1)^2$, so we fix
$S=23/4$, obtaining

$$
\Delta_2:=\left.\Delta\right|_{S={23}/{4}}=\displaystyle{\frac{1}{16}\,\frac{2116\,{a}^{3}-8076\,{a}^{2}-17871\,a-2116}{\left(
a-4 \right) \left( 1+2\,a \right) ^{2}}}.
$$
In order to avoid the terms which are perfect squares and to have
an algebraic expression, we consider
$$
\Delta_3:=
\displaystyle{\left(\frac{4(1+2a)(a-4)}{46}\right)^2}\,\Delta_2=
(a-4)\left(a^3-\frac{2019}{529}a^2-\frac{777}{92}a-1\right),
$$
and we try to find rational values of $a$ on the quartic
$k^2=\Delta_3$, that is
\begin{equation*}
k^2=(a-4)\left(a^3-\frac{2019}{529}a^2-\frac{777}{92}a-1\right).
\end{equation*}
In order to apply Lemma~\ref{lemaxarles} to the above equation we
perform the following translation $A=a-{4135}/{2116},$ obtaining
the new elliptic quartic
\begin{equation}\label{qua}
K^2={A}^{4}-{\frac {36024561}{2238728}}\,{A}^{2}-{\frac
{38272338}{ 148035889}}\,A+{\frac
{1009624858257249}{20047612231936}},
\end{equation}
which, by Lemma \ref{lemaxarles}, is isomorphic to the elliptic
cubic
\begin{equation}\label{eqXY}
{Y}^{2}={X}^{3}-{\frac {1288423179}{71639296}}\,X+{\frac
{8775405707427}{ 303177500672}}.
\end{equation}

Consider the additive group of rational points on the elliptic
curve (\ref{eqXY}):
$$
E(\Q)=\left\{(X,Y)\in\Q\times\Q\,\mbox{ satisfying
(\ref{eqXY})}\right\} \cup\cal{O}
$$
where $\cal{O}$ is the usual point at infinity acting as a neutral
element, by using MAGMA  we have found the following rational point
 $R=(\frac{18243}{8464},\frac{81}{184})$, which lies
on the compact oval of (\ref{eqXY}), see Figure~1.
\begin{center}
\includegraphics[scale=0.30]{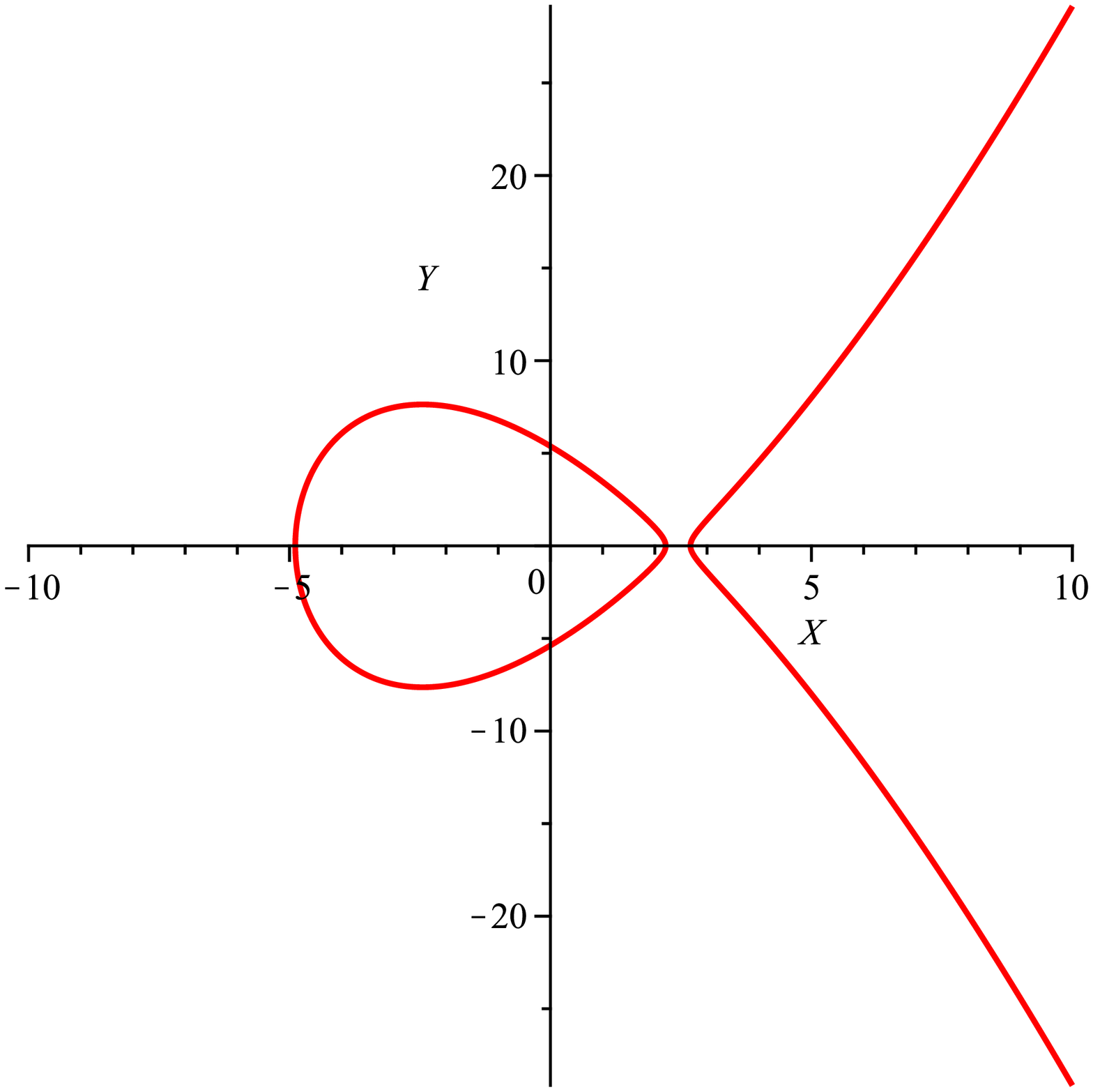}
\end{center}
\begin{center}
Figure 1: The elliptic curve ${Y}^{2}={X}^{3}-{\frac
{1288423179}{71639296}}\,X+{\frac {8775405707427}{ 303177500672}}$.
\end{center}
\begin{center}
\includegraphics[scale=0.30]{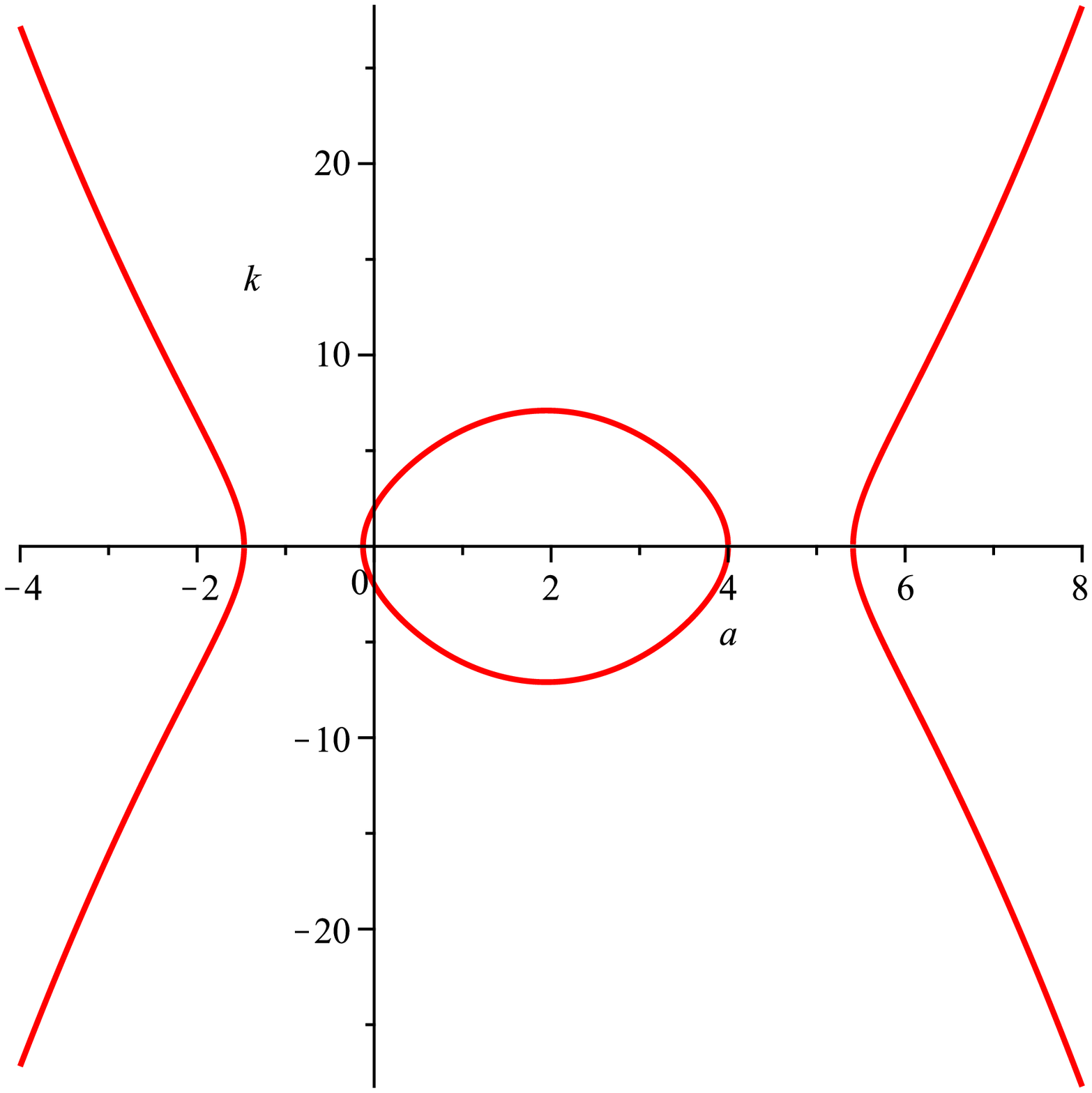}
\end{center}
\begin{center}
Figure 2: The elliptic quartic
$k^2=\Delta_3=(a-4)\left(a^3-\frac{2019}{529}a^2-\frac{777}{92}a-1\right).$
\end{center}

It is easy to check that $R$ is not in the torsion of $E(\Q).$ In
fact, $E(\Q)$ has trivial torsion. Then, as has been explained in
Section~\ref{grouplaw}, we know that the set of rational points $\{k
R\}_{k\in\Z}$, obtained by the secant-tangent chord process, fill
densely the full elliptic curve~\eqref{eqXY}. So, their images,
trough the transformations given in our proof form a dense set of
rational points on the quartic~\eqref{qua}. The corresponding
projections give a set of rational values of $a$ in
$(-\infty,a_4]\cup[a_3,a_2]\cup[a_1,+\infty)$ such that the
corresponding map $F_a$ has 9-periodic rational
points\footnote{Notice that not all the rational points of
on~\eqref{eqXY}  are good seeds for obtaining the density of values
of $a$ in $[a_1,+\infty)$ by the secant-tangent chord process. For
instance, although  the rational point
$(X,Y)=\left({23947}/{8464},{1781}/{2116}\right)$ gives
$$
(a;x(a),y(a))=\left(\frac{50025}{6344};\frac{4231448}{8351929},
\frac{175168575}{33407716}\right),
$$
which is a counterexample to Zeeman's Conjecture 2, it is not useful
for our purposes because it lies on the unbounded connected
component of \eqref{eqXY}.}. Here $a_4<a_3<a_2=4<a_1$ are the four
roots of $\Delta_3$, see Figure~2. Of course, these periodic points
can not have both coordinates positive when $a$ belongs to the first
two intervals. For instance one of two points corresponding to $R$
is
\[
(a;x(a),y(a))=\left(\frac {391}{370};\frac{28543}{4224},
-\frac{4255}{4224}\right).\]

Finally we prove that the rational points found on the
quartic~\eqref{qua} corresponding to values of $a\in[a_1,+\infty)$
give rise to positive rational 9-periodic orbits. Observe that for
$S=23/4$, the value of $P$ given by equation (\ref{eqP}) is
$$
P=\displaystyle{\frac{27}{4}\,\frac{a(4a+23)}{(a-4)(1+2a)^2}}.
$$
Hence for all values of $a\ge a_*,$ the value $P$ is positive and
therefore the corresponding values of $x$ and $y$ are also
positive. Hence the theorem follows. \qed

Notice that our proof of Theorem~\ref{teoA} searches positive
rational points $(a;x,y)$ in $S_a$, such that $x+y=23/4$. Hence a
first necessary condition for their existence is that both curves
intersect in the first quadrant of the $(x,y)$-plane. It is easy to
prove that when $a\ge a_*$ this only happens when $a\ge a_1>a_*.$

Some values of $a$ in $(a_*,a_1)$ with not large numerators and
denominators\footnote{We have obtained them by computing several
convergents of the expansion in continuous fractions of same points
in the interval.} are:
\begin{align*}
&\frac{14243}{2632},\,\frac{14335}{2649},\,\frac{18675}{3451},\,\frac{197021}{36408},\,\frac{
216459}{40000},\,
\frac{333060}{61547},\\
&\frac{11034}{2039},\,\frac {21976}{4061},\,
\frac{60641}{11206},\,\frac{96005}{17741},\,\frac{96860}{17899}.
\end{align*}
Notice that $a_1-a_*\simeq 0.21\times 10^{-5}.$ For the above values
we have not been able to find rational points on the corresponding
elliptic curve~\eqref{9p}. Nevertheless it is not difficult, by
using again MAGMA or SAGE, to compute the root numbers associated to
these elliptic curves, obtaining $-1$ for the values of the first
row and $+1$ for the ones of the second row. Hence, if one believes
that the Birch and Swinnerton-Dyer Conjecture is true (see for
example \cite{C}, Conjecture 8.1.7.), as most people do, or more
concretely, on the trueness of the so-called Parity Conjecture (see
for example \cite{C}, section 8.5), we obtain that for the first six
values the corresponding curve would have rational points (not
necessarily with positive coordinates).

\begin{nota}
The elliptic curve given by the equation \eqref{eqXY} has minimal
model
$$ Y^2 + XY + Y = X^3 - X^2 -
994154X + 376423337.$$ A computation with either SAGE or MAGMA
reveals that this curve has rank $3$. As far as we know, this is the
highest known rank for the base elliptic curve of a family of
elliptic curves with a $9$ torsion point parameterized by such
curve, see \cite{W} or \cite[App. B.5]{R}.
\end{nota}

Next, we show the non-existence of rational 9-periodic points for a
some~$F_a, a>a_*.$

\begin{propo}\label{propoc}
(i) For $a\in\{6, 8\}$ there are no initial conditions in
$\Q\times\Q$ such that the sequence~\eqref{Lynesseq} has period 9.

(ii) For $a=9$ there are no initial conditions in $\Qm\times\Qm$
such that the sequence~\eqref{Lynesseq} has period 9, but there are
infinitely many in $\Q\times\Q$.
\end{propo}

\begin{proof} Consider the elliptic curves~\eqref{9p} for $a=6,8$
and $9$. By using MAGMA we can identify them in the list of Cremona
curves (\cite{Cr}) as numbers 17670bb1, 122094bl3 and 118698i1,
respectively. From the list we know that for the first two cases the
rank is 0. Hence they have no  rational points apart of the
$9$-torsion points. In the case $a=9$ we have that the rank of the
elliptic curve is $1$, so there are infinitely many rational points
on it. But a generator of the free part is $(-3/70 , -1273/105)$,
which has not  positive coordinates. This implies that no rational
point in the elliptic curve has positive coordinates as well.
\end{proof}

The result (i) for $a=6$ of the above proposition can also be proved in
different ways. It is possible to use the 2-descent method, or, even
better, since our curves have $9$-torsion points and, hence,
$3$-torsion points, one can instead use the descent via 3-isogeny
method. It is also possible to use an indirect method: showing that
the value of its $L$-function at $s=1$ is not zero. Then, by the
known true cases of the Birch and Swinnerton-Dyer Conjecture (see
\cite{C}, Theorem 8.1.8), we can deduce its rank is 0. By using
MAGMA or SAGE, one shows that $L(E,1)\simeq 6.10077623$, and hence
$E$ has rank $0$. The same argument applies for $a=8.$

By using Theorem~\ref{nft}, together with the known example of an
elliptic curve of rank $4$ with a torsion point of order $9$, see
\cite{W},  we get an example of rank $4$ for the curves $C_{a}$.
This curve gives also a counterexample of Zeeman's Conjecture~2 with
many (positive) rational periodic orbits.

\begin{corol} For $a=408/23$, the curve $C_{a}$ has rank $4$ over
$\Q$, that is $C_a{(\Q})\cong\Z^4\times \Z/9\Z.$ Moreover the
non-torsion points  are  generated by
\begin{align*}
&\left(\frac{15708}{38617}, \frac{1275}{4346}\right),
\left(\frac{117348775936}{1130069373},
\frac{17875982344}{22803541107}\right),\\
&\left(\frac{-5313}{5186},
\frac{199644}{17}\right),\left(\frac{96539240}{980237},
\frac{892914}{1232041}\right).
\end{align*}
\end{corol}

\section{Period $\mathbf 5$ points}\label{periode5}

When $a=1,$
the map $F_1(x,y)=(y,(1+y)/x)$ is the celebrated, globally
5-periodic, Lyness map. We define $E_h:=C_{1,h}$ as the family of
elliptic curves
\[
E_{h}:=\{(x+1)(y+1)(x+y+1)-h\,xy\,=\,0\}.
\]
It is clear that all them are filled of 5-periodic orbits. Following
\cite{BR} we know the torsion of the group of rational points of
each of the curves, $E_h(\Q)$ is either ${\Z}/{5\Z}$ or
${\Z}/{10\Z}$. In that paper it is proved that $(1,4)\in E_{15}$ or
that $(1,3)\in E_{40/3}$ and that they are not in the corresponding
torsion subgroups. Hence the corresponding elliptic curves have
infinitely many positive rational points and moreover the ranks of
both groups are greater or equal than $1$. In \cite[Prob. 1]{BR} the
authors ask the following question: Is the rank of every $E_h(\Q)$
for $h\in\Q,$ $h>h_c^+$ always positive? Here
$h_c^+=(11+5\sqrt{5})/2\simeq 11.09$ is the value where $E_h$ starts
to have points in the first quadrant. Next result proves that the
answer to the above question is negative.

\begin{propo}\label{propod}
There exist rational values of $h,$ $h>h_c^+$, such that $E_{h}(\Q)$
has rank $0$.
\end{propo}

\begin{proof} We use the same method  that in the proof of
Proposition~\ref{propoc}. Consider $E_h$  with $h=13.$ Its
corresponding number in Cremona's list (\cite{Cr}) is 325e1,
obtaining rank 0 and torsion $\Z/5\Z$. Hence the result follows.
\end{proof}

\begin{nota}
(i) Arguing as in the previous proof, for  the values $h=12,14$ and
$15$ we obtain ranks 0,0 and 1,  and torsions $\Z/10\Z,$ $\Z/5\Z$
and $\Z/5\Z$, respectively. In fact, in the $100$ first integers
values of $h>11$, there are $48$ values with rank 0 (and $43$ values
with rank 1, and $9$ values  of rank 2).

(ii) For the case $h=12$, one can also use the existence of a point
with 2-torsion to give an easy  proof that the rank of $E_{12}(\Q)$
is 0, by using the descent  via 2-isogeny method.

(iii) Recall that the Lyness map can be written as $P\rightarrow
P+Q$, where $Q=[1:0:0]$ and this point has 5-torsion. Thus  when the
initial point $P_0$ has 2-torsion we obtain that
\[
P_0+(2m+1)Q=(2m+1)(P_0+Q),\quad m\ge0
\]
and so the odd multiples of $P_0+Q$ by the secant--tangent chord
process coincide with the points generated by the Lyness map.

(iv) The case $h=12$  is also interesting because, although the rang
of $E_{12}$ is zero, there is a unique rational (integer) orbit
corresponding to half of the torsion points. This orbit is given by
the initial conditions $x_0=1$, $x_1=1$ and is
\[
1,1,2,3,2,1,1,\ldots
\]
As it is explained in item (iii), this situation happens because the
point $(1,1)$ has 2-torsion.

(iv) By using once more the universality of the level sets of the
Lyness curves $C_{a,h}$ and the results given in \cite{W}, where
four examples of elliptic curves with torsion $\Z/5\Z$ and rank 8
appear, we have obtained the following values of $h$:
\[
\frac{2308482}{11325881},\quad\frac{283551345}{1294328864},\quad
\frac{2489872}{6474845},\quad\frac{9645545}{74914011}
\]
(and also the ones corresponding to $-1/h$) for which the rank of
$E_h$ is $8$. Since all these values are in $(0,1)$ none of them
satisfies $h>h_c^+.$

(v) By using the same method than in the previous item we know
 that $E_{\tilde h}$ and $E_{-1/\tilde h}$ for $\tilde
 h=12519024/498355\simeq 25.12>h_c^+$ have rank 7.
\end{nota}

\section{Single Lyness recurrences with many periods}

Looking at Theorem~\ref{teo} it is natural to wonder if there exists
some rational  fixed value of $a$ for which there are rational
initial conditions such that the corresponding sequences generated
by \eqref{Lynesseq} has all the allowed periods
$\{1,2,3,5,6,7,8,9,10,12\}$ given the theorem. It is easy to see
that the answer is no, because period 5 (resp. 6) only happens when
$a=1$ (resp. $a=0$). Moreover we will prove in next proposition that
rational sequences with prime periods $1$ and $12$ or $2$ and $12$
can not coexist for a given value of $a$. On the other hand, as it
is shown in Table~2, for each $a\ne1$, the point $(-1,-1)$ always
gives rise to a 3-periodic orbit for $F_a$.

Hence the interesting  problem is to  find values of $a\in\Q$ such
that $F_a$ has one of the following sets of prime periods realized
for rational periodic points:

\[\mathcal{P}_1=\{1,2,3,7,8,9,10\}\quad \mbox{or}
\quad\mathcal{P}_2=\{3,7,8,9,10,12\}.\]

Table~2 shows that for $a=21/37$ the set of periods of $F_a$ is
$\mathcal{P}_2$. The search of rational points with periods $9$ and
$10$ is quite involved. We give some details of how we have obtained
them in the end of this section.

This table also shows that for the integer value $a=20$ the set of
achieved periods is $\mathcal{P}_1\setminus\{2\}$ and it is easy to
see that period 2 does not appear. To find   explicit rational
values $a$, candidates to have  as set of periods with rational
entries the full set $\mathcal{P}_1,$ is not very difficult. First
we give a rational parametrization of the values of $a$ for which
rational fixed and period 2-points appear. After, among these
values, we search for ``small'' values for which the root numbers of
the elliptic curves corresponding to the points of period $7,8,9$
and $10$ is $-1$. We have obtained the list
\[
\frac{88401}{18496},\quad
\frac{136353}{23104},\quad\frac{139971}{36100},\quad\frac{9633}{12544},\quad
\frac{17301}{19600},\quad
\frac{157521}{92416},\quad\frac{31003}{39204}.
\]
For all these values of $a$,  if one believes once more in the
trueness of the Birch and Swinnerton-Dyer Conjecture or of the
Parity Conjecture, the set of periods of the corresponding $F_a$
should be $\mathcal{P}_1$. Unfortunately for none of these values of
$a$ we have  been able to find  explicitly periodic points of all
the periods.

\vspace{0.2cm}

\begin{center}
\begin{tabular}{|c||c|c|}
  \hline
 Period &$(x_0,x_1)$ for  $a=20$ &$(x_0,x_1)$ for   $a=21/37$ \\
 \hline
  1 & $(5,5)$ & - \\
 2 & - & - \\
 3 & $(-1,-1)$ & $(-1,-1)$ \\
 7 & $(-\frac{11}3, -\frac{35}{32})$& $(\frac{455}{1679}, -\frac{9394}{6693})$\\
 8 & $(-\frac{95}2, -\frac{31}{12})$ & $( \frac{221}{14}, -\frac{645}{658})$ \\
 9 &$(\frac{5}{166}, -\frac{95}{12})  $ &
$(-\frac{2719003411664}{4342282089993},
\frac{25886110233337}{102273997737527} )$\\
 10 & $(-\frac{60905}{253889}, -\frac{5756625}{291104}) $
 & $(  \frac{1657822032572550308388507}{4355431052669166166335275},
 -\frac{1803238432370002727833401}{2680435796120980996248701})$ \\
 12 & - & $( -\frac{51}{35}, -\frac{32}7)$\\
  \hline
\end{tabular}
\vspace{0.5cm} \nobreak\\ Table 2. Values of $a$ and rational
initial conditions for the recurrences~\eqref{Lynesseq}\\ with
periodic sequences of several periods.
\end{center}

\vspace{0.2cm}

\begin{propo}\label{incomp} Given $a\in\Q,$ the Lyness map $F_a$ has
not simultaneously rational periodic points with prime periods 1 and
12, or 2 and 12.
\end{propo}

\begin{proof} From the results of Subsection~\ref{nonelli} we know
that these couple of periods do not coexist when the corresponding
curves $C_{a,h}$ are not  elliptic curves. So, from now one we can
assume that the 12 periodic points lie on an elliptic curve
$C_{a,h}$. First of all, given $a\in \Q$, it is easy to obtain that
there is a rational periodic point with period 1 if and only if
$4a+1$ is a square in $\Q$ and that there is one with period 2 if
and only if $4a-3$ is a square in $\Q$. By using the expressions of
$k Q$ given in Section 2.1 (or by using the known results for the
Tate curve \cite{AM}), we get that the explicit parameterization of
the values of $a$ and $h$ for which $Q$ can have order 12 is
\[
a=\frac{2t(1+t)}{3t^2+1} \quad\mbox{and,}\quad
h=-\frac{(t-1)^2(t^2+1)}{t(1+t)(3t^2+1)},\] for some $t\ne 0,\pm1$.
Hence, to have a rational periodic points with periods 1 and 12, we
need rational numbers $t$ such that $ \frac{8t(1+t)}{3t^2+1}+1$ is a
square in $\Q$. We will show that this only happens for the values
$t=0$ and $1$, which do not give prime period 12. Multiplying by
$(3t^2+1)^2$, we need to search for rational solutions of the
equation $z^2=(3t^2+1)(11t^2+8t+1)$. A standard change of variables,
similar to the one used used in Section~\ref{periode9}, shows that
this genus one curve is isomorphic to the elliptic curve with
corresponding number in the Cremona's list equal to 15a8, which has
only four rational points. These points correspond to the points
with $t=0$ and $t=-1$.

Similarly, to have  rational periodic points with periods 2 and 12,
we need rational numbers $t$ such that $ \frac{8t(1+t)}{3t^2+1}-3$
is a square in $\Q$. Arguing  as in the previous case we arrive to
the equation $z^2=(3t^2+1)(-t^2+8t-3)$ which number in Cremona's
list is 39a4 and has only the two rational points corresponding to
$t=1$. So the result follows.
\end{proof}

\subsection{Searching 9 and 10 rational periodic points for
$F_{21/37}$}

In order to find 9 or 10   rational periodic points for $F_{21/37}$,
one cannot just naively search for points, since their coordinates
are too big. So we use the following strategy. First, with the same
formulas  that in the proof of Theorem~\ref{nft}, we transform the
equations $C_{21/37,h},$ for  their corresponding
$h=-\frac{16528}{28749}$ and $h=-\frac{296}{609},$ to a Weierstrass
equation. We need to find non-torsion points on these elliptic
curves. Since the Weierstrass equations have too big coefficients,
we apply $2$-descent procedure with MAGMA in order to get an
equivalent quartic equation with smaller coefficients for the
corresponding elliptic curves. We get that they are equivalent
respectively to
$$y^2 = -57376476x^4 + 66683940x^3 + 800552377x^2 - 118125576x +
209901456$$ and
$$ y^2 = -7734191x^4 + 116312038x^3 + 178646017x^2 - 246594696x - 138820464.$$
These are still not sufficient simple to be able to find
(non-torsion) rational points, so we do another transformation. In
the first case, we apply $4$-descent in order to obtain another
form, this case as intersection of two quadrics (in the projective
space). We get the equations
$$ 15X^2 + 104XY - 16Y^2 + 12XZ - 42YZ + 18Z^2 +
62XT - 14YT - 8ZT + 22T^2 = 0, $$ $$123X^2 - 460XY - 233Y^2 - 24XZ -
398YZ - 153Z^2 + 122XT - 320YT + 688ZT + 321T^2=0. $$ Finally, an
easy search finds the point given in projective coordinates by
$[-\frac{10}{21} : \frac9{14} : \frac{19}6 : 1]$. Using the
transformation rules given by MAGMA one gets the corresponding point
in $C_{21/37,h}$, shown in Table~2. For the second equation,
corresponding to the 10 rational periods, we directly transform the
quartic equation to an intersection of two quadrics, and then  we
apply an algorism due to Elkies (\cite{E}) to search for rational
points in this type of curves (as implemented in MAGMA).

\section*{ Acknowledgements} The authors are partially supported by
MCYT through grants MTM2008-03437 (first author),
DPI2008-06699-C02-02 (second author) and MTM2009-10359 (third
author). The  authors are also supported by the Government of
Catalonia through the SGR program.

\end{document}